\documentclass{amsart}
\usepackage{amsmath, amsthm, amssymb}
\usepackage{verbatim}
\usepackage{tikz}

\renewcommand{\d}{\partial}

\newcommand{\eps}{\epsilon}
\newcommand{\veps}{\varepsilon}
\newcommand{\vphi}{\varphi}
\newcommand{\al}{\alpha}

\newcommand{\ga}{\gamma}
\newcommand{\de}{\delta}
\newcommand{\la}{\lambda}
\newcommand{\om}{\omega}

\newcommand{\vrh}{\varrho}

\newcommand{\ka}{\kappa}

\newcommand{\Ga}{\Gamma}
\newcommand{\Om}{\Omega}

\newcommand{\cE}{\mathcal{E}}

\newcommand{\cF}{\mathcal{F}}
\newcommand{\cH}{\mathcal{H}}

\newcommand{\cS}{\mathcal{S}}
\newcommand{\supp}{\mbox{supp }}

\newcommand{\bR}{\mathbb{R}}
\newcommand{\bC}{\mathbb{C}}

\newcommand{\vpi}{\varpi}

\newtheorem{thm}{Theorem}
\newtheorem{prop}[thm]{Proposition}
\newtheorem{lem}[thm]{Lemma}
\newtheorem{cor}[thm]{Corollary}

\theoremstyle{definition}
\newtheorem{defn}[thm]{Definition}
\newtheorem{remark}[thm]{Remark}

\newtheorem{expl}[thm]{Example}

\numberwithin{thm}{section}
\numberwithin{equation}{section}

\renewcommand{\[}{\begin{equation}}
\renewcommand{\]}{\end{equation}}

\newcommand{\wed}{\wedge}

\title[Monge-Amp\`ere equations for measures dominated by capacity]{Continuous solutions to Monge-Amp\`ere equations on Hermitian manifolds for measures dominated by capacity}
\author{S\l awomir Ko\l odziej and Ngoc Cuong Nguyen} 
\address{Faculty of Mathematics and Computer Science, Jagiellonian University, \L ojasiewicza 6, 30-348 Krak\'ow, Poland}
\email{slawomir.kolodziej@im.uj.edu.pl}
\address{Department of Mathematical Sciences, KAIST, 291 Daehak-ro, Yuseong-gu, Daejeon 34141, South Korea}
\email{cuongnn@kaist.ac.kr}

\begin{document}
	\maketitle
	\begin{abstract} We prove the existence of a continuous quasi-plurisubharmonic solution to the Monge-Amp\`ere equation on a compact Hermitian manifold for a very general  measre on the right hand side. We admit  measures dominated by capacity in a certain manner, in particular, moderate measures studied by Dinh-Nguyen-Sibony.  As a consequence, we give a characterization of measures admitting H\"older continuous quasi-plurisubharmonic potential, inspired by the work of Dinh-Nguyen.
	\end{abstract}
	
\bigskip
\bigskip
	
\section{introduction}


\bigskip

Let $(X,\om)$ be a compact Hermitian manifold of dimension $n$. The study of the  complex Monge-Amp\`ere equation in this setting was initiated by Cherrier \cite{cherrier87},
and  the counterpart of the Calabi-Yau theorem \cite{Y} on compact Hermitian manifolds was  proven by Tosatti and Weinkove \cite{TW10b}.
Later Dinew and the authors, in a series of papers \cite{DK12}, \cite{KN1,KN4,KN2},  obtained weak continuous  solutions for more general densities on the right hand side of the equation, by extending the pluripotential methods employed before on the  K\"ahler manifolds. In this paper we deal with yet more general measures on the right hand side.

 If $\om$ is K\"ahler, then the first named author obtained in \cite{K98,K03} the unique continuous $\om$-plurisubharmonic ($\om$-psh for short) solution to the complex Monge-Amp\`ere equation with the right hand side being a measure in one of the  classes $\cF(X, h)$ satisfying a bound in terms of the Bedford-Taylor capacity and a weight function $h$ (the precise definition is given in the next section). 
We prove here the generalization of this result to Hermitian manifolds.

\begin{thm}\label{thm:intro-A} Let $\mu \in \cF(X,h)$ be such that  $\mu(X)>0$. Then, there exists a continuous $\om$-psh  function $u$ and a constant $c>0$ solving the equation
	\[\notag
	(\om + dd^c u)^n = c\; \mu.
	\]
\end{thm}

If we assume further that the right hand side is strictly positive and absolutely continuous with respect to the Lebesgue measure, then we prove a stability of solutions  and their uniqueness extending the main theorem of \cite{KN2}, (Theorem~\ref{thm:x-stability}).
Our method is  adaptable to  the Monge-Amp\`ere type equations  \cite{N16}. As a consequence, we get the existence and uniqueness of continuous $\om$-psh solutions of these equations  (Corollary~\ref{cor:max}).

The families of  measures which belong to $\cF(X, h)$, for some $h$, include those having densities in $L^p , p>1,$ or even broader Orlicz spaces, but also measures singular with respect to $\om ^n$,
for instance smooth forms on  totally real submanifolds (see  e.g. \cite{K98}, \cite{BJZ},  \cite{viet16}). We shall distinguish
 classes $\cH(\tau)$ which are unions (over $C>0$) of $\cF(X,h_1)$  with $h_1(x) = C x^{n\tau}$ and fixed  $\tau>0$; and $\cF(X, h_2)$ with $h_2(x)= C e^{\al x}$ for some $C, \al>0$. The latter was introduced by
Dinh, Nguyen and Sibony \cite{DNS}, who called the measures in this class (the union over $C>0, \al >0$) {\em moderate}.
They   proved that any measure locally dominated by the Monge-Amp\`ere measure of a H\"older continuous psh function is moderate. 

Later, Dinh and Nguyen \cite{DN14} characterized the measures locally dominated by the Monge-Amp\`ere measure of a H\"older continuous psh function 
via the associated functionals acting on $PSH(\om)$ (the set  of all $\om$-psh functions) when $\om$ is K\"ahler. 
In the last section we  give a similar description in the  Hermitian setting. Let us define
$$
	\cS:= \left\{ v\in PSH(\om): -1\leq v\leq 0, \;\sup_X v=0\right\}.
$$
 Let $\mu$ be a positive Radon measure on $X$ and  $\hat{\mu }: PSH(\om) \to \bR$ the  associated functional  given by
 $$
 	\hat{\mu }(v)= \int_X v d\mu.
 $$

\begin{thm}\label{thm:intro-B} The measure $\mu$ belongs to $\cH(\tau)$ and $\hat{\mu }$ is H\"older continuous with respect to $L^1$-distance on $\cS$ if and only if there exists a H\"older continuous $\om$-psh function $u$ and a constant $c>0$ solving $
	(\om + dd^c u)^n = c \;\mu.$
\end{thm}

Notice that the H\"older continuity of $\hat{\mu }$ on the larger subset $\{v\in PSH(\om): \sup_X v =0\}$ implies the $\cH(\tau)$ property and the H\"older continuity on $\cS$ (Propostion~\ref{prop:dn-moderate}). 
The  latter properties are independent. The examples \cite[Example 5.5]{DDGHKZ} or \cite[Example~2.5]{DN14} belong to $\cH(\tau)$ for every $\tau>0$, but they do not admit H\"older continuous potentials. On the other hand,  the well-known conjecture of Dinh, Nguyen, Sibony \cite[Problem~1.5]{DN14} predicted that the moderate property implies the H\"older continuity of the Monge-Amp\`ere potential or equivalently the H\"older continuity on $\cS$ of the functional associated to this measure.

As it was shown in \cite{KN4} (inspired by \cite{DDGHKZ}) the existence of H\"older continuous solution is a local problem. We apply Theorem~\ref{thm:intro-B} to get main results of \cite{hiep} and \cite{viet16} in the Hermitian setting. For example, this gives a H\"older continuous $\om$-psh potential for a  smooth volume form  of a compact smooth real hypersuface in $X$. 

Let us indicate some  motivations behind the study of  the Monge-Amp\`ere on Hermitian manifolds with measures on the  right hand side. 
Unlike in the K\"ahler case, one solves the equation not only for a function but also for  a constant on the right hand side. The range of those constants for a given
manifold seems to have a geometrical meaning. It comes up in  constructions of  $\om$ - psh functions with logarithmic poles like in \cite{TW12}, \cite{N16} (where
one solves the equation for approximants of Dirac measures);
in connection to problems involving holomorphic Morse inequalities (see \cite{KTo}), and others.
The parabolic Monge-Amp\`ere on a Hermitian manifold, the Chern-Ricci flow, is recently intensively studied (see \cite{gill11,gill13}, \cite{FTWZ}, \cite{Ni17}, \cite{TW12b,TW12a} \cite{TWYang15}, \cite{Zh17}). The flow is expected to play an important role in the classification of complex surfaces. In the context of parabolic equations the pluripotential estimates are also useful. For example, To \cite{To18} (independently, Nie \cite{Ni17} in particular cases) used results in \cite{DK12} and \cite{N16} to prove a conjecture by Tosatti and Weinkove  \cite{TW12b}. 
The geometric applications of pluripotential theory on Hermitian manifolds are discussed at length in surveys by Dinew \cite{Di16,Di19}.

 Another new topic is the complex dynamics on compact Hermitian manifolds. There the measures having interesting properties are often singular with respect to the volume form. In a recent paper  Vu \cite{viet19} showed that for any holomorphic dominant endomorphism $f$ of $X$ there exists an equilibrium measure $\mu_f$ associated to $f$. The understanding of  this measure is a central problem in complex dynamics (as in the K\"ahler setting). By  \cite[Theorem~1.1]{viet19} and Theorem~\ref{thm:intro-B} one gets that $\mu_f$ admits a H\"older continuous $\om$-psh potential. We refer  to \cite{DiF} for results on  the push-forwards of  measures by dominant meromorphic maps between complex manifolds.

\bigskip

{\em Acknowledgement. \rm} The authors are partially supported by NCN grant  2017/
27/B/ST1/ 01145. The second author is also partially supported by the start-up grant G04190056 of KAIST.

\section{preliminaries}
In this section we recall and extend some results from \cite{KN1,KN2,KN3}. 
Their statements are often more technical than the counterparts in the K\"ahler setting \cite{K05}.

Let $h : \mathbb R_+ \rightarrow (0, \infty ) $ be an increasing function such that
\[
\label{eq:admissible}
	\int_1^\infty \frac{1}{x [h(x) ]^{\frac{1}{n}} }  \, dx < +\infty.
\]
In particular, $\lim_{ x \rightarrow \infty} h(x) = +\infty$. Such a function $h$ is called
{\em admissible}. In what follows we often omit to stress that $h$ is admissible.
If $h$ is admissible, then so is $A_2 \, h (A_1x)$ for every  $A_1,A_2 >0$.
Define
\[
	F_h(x) = \frac{x}{h(x^{-\frac{1}{n}})}.
\]

Recall that the  analogue of Bedford-Taylor capacity on compact complex manifolds is
$$	cap_\om(E) := \sup\left\{\int_E \om_w^n : w\in PSH(\om), 0\leq w\leq 1\right\},
$$ where $PSH(\om)$ is the set of $\om$-psh functions   on $X$ and $ \om_w^n := (\om +dd^c w)^n$.
This capacity is equivalent to the Bedford-Taylor capacity \cite{BT82} defined locally (see \cite[page 52-53]{K05}).  

Let $\mu$ be a positive Radon measure satisfying
\[
\label{eq:dominate}
	\mu(E) \leq F_h( cap_\om (E)),
\]
for any Borel set $E \subset X$ and  some  $F_h$. 

Let us denote by $\cF(X,h)$ the set of all measures that are dominated by the capacity $cap_\om$ in the sense of \eqref{eq:dominate} for some admissible $h$. 

Some particular  families of measures which satisfy \eqref{eq:dominate} were mentioned in Introduction.
 Another fairly general family is given in the following example. Note that  these measures are often singular with respect to the Lebesgue measure and their potentials  may not be H\"older continuous.

\begin{expl} (\cite{KN5}) Let $\mu$ be a positive Borel measure such that it is locally dominated by Monge-Amp\`ere measures of continuous plurisubharmonic functions whose modulus of continuity $\vpi(t)$ satisfy the Dini-type condition
\[
	\int_0^1 \frac{[\vpi(t)]^\frac{1}{n}}{t |\log t|} dt <+\infty.
\] 
Then, $\mu \in \cF(X,h)$ for some admissible function $h$.
\end{expl}


Let us fix a  finite covering of $X$:
\[\label{eq:cover}\{B_j(s)\}_{j\in J} \quad \text{where}\quad B_j(s):= B(x_j,s)
\]  
 is the coordinate ball centered at $x_j$ of radius $s>0$. Take $s$ so small that $B(x_j,3s)$, $j\in J$, are still coordinate balls. Let $\chi_j$ be the partition of unity subordinate to $\{B_j(s)\}_{j\in J}\}$. 
The first observation is that if $\mu$ satisfies \eqref{eq:dominate} on $X$, then in each chart $B_j(3s)$ the same property holds for subsets of the smaller ball.

\begin{lem}\label{lem:restriction} Let $\mu\in \cF(X,h)$. Then, for every compact $K \subset B_j(s) \subset \Om:= B_j(3s)$, $j\in J$,
\[\label{eq:vc-local}
	\mu(K) \leq F_{h_{0}}\left(cap(K, \Om)\right).
\]
for an admissible function $h_0$ depending only on $h, \om, X$ and $\Om$, where $cap(K, \Om)$ is the relative capacity of Bedford and Taylor \cite{BT82}.
\end{lem}

\begin{proof}  The proof follows by the monotonicity of $h$ and the fact that
\[\notag
	cap_\om(K) \leq C_1 cap(K, \Om),
\]
where $C_1$ is a uniform bound for plurisubharmonic functions on $B(x_j, 3s)$ such that $v_j = 0$ on $\d \Om$ and $dd^c v_j \geq \om$ in $\Om$ (see \cite[page 53]{K05}). Thus, we can take $$h_0(x)= \frac{1}{C_1} h(C_1^{-\frac{1}{n}} x).$$
The proof is completed.
\end{proof}

The second observation is the following.

\begin{lem}\label{lem:app-seq} Let $\mu \in \cF(X,h)$. Let $\mu_{U_j}$ be the restriction of $\chi_j\mu$ to the local coordinate $\Om _j = B(x_j, 3s) \subset \bC^n$, 
where $U_j= B(x_j,s)$. Let $\rho_\veps$ be the standard smoothing kernel on $B(0, 3s)$. Then, 
\[\label{eq:app-seq}
	\mu_\veps(z):= \sum_{j\in J}  \mu_{U_j} *\rho_\veps(z-x_j)
\]
is the sequence of smooth measures which converge weakly to $\mu$ as $\veps$ tends to $0$. Moreover,  $\mu_\veps \in \cF(X, h_0)$ for an admissible function $h_0$ when $\veps$ is small enough.
\end{lem}

\begin{proof} Since the cover is finite, it is enough to show that each smooth measure of the right hand side belongs to $\cF(\Om, h_0)$ for an admissible function $h_0$. By Lemma~\ref{lem:restriction} it follows that $\mu_{U_j} \in \cF(\Om, h_0)$. Thanks to \cite[Eq.(3.5.1)]{K98} the convolutions with smoothing kernels preserve the inequality \eqref{eq:vc-local} when $\veps$ is small enough. 
\end{proof}

We recall the basic result in \cite{KN1}. Let $B>0$ be a constant such that
$$	-B \om^2 \leq 2n dd^c \om \leq B \om^2, \quad -B \om^3 \leq 4n^2 d\om \wed d^c \om \leq B \om^3.
$$
\begin{thm}
\label{thm:kappa}
Fix $ 0 < \varepsilon <1$. Let $ \varphi, \psi \in PSH (\omega)\cap L^\infty(X)$ 
be such that $\varphi \leq 0$, and $ -1 \leq \psi \leq 0$. Set 
$m(\varepsilon) = \inf_X [ \varphi - (1-\varepsilon) \psi]$, and 
$
\varepsilon_0:	\frac{1}{3}\min\{
	\varepsilon^n, 
	\frac{\varepsilon^3}{16 B}, 
	4 (1-\varepsilon) \varepsilon^n, 
	4 (1-\varepsilon)\frac{\varepsilon^3}{16 B} \}
$.
Suppose that 
$\omega_\varphi^n \in \cF(X,h)$. Then, for $0<t< \varepsilon_0$,
\[\label{eq:kappa}
	t \leq \kappa\left[ cap_\omega ( U(\varepsilon, t))\right],
\]
where $U(\varepsilon, t ) = \{ \varphi < (1- \varepsilon) \psi + m(\varepsilon) + t \}$, 
and the function $\kappa $ is defined on the interval $(0,cap_\omega(X))$ by the formula
\[\label{eq:kappa2}
	\kappa \left ( s^{-n} \right) 		4\, C_n  \left \{
			\frac{1}{ \left [ h ( s )\right]^{\frac{1}{n}} } 
			+ \int_{s}^\infty  \frac{dx}{x \left[ h (x) \right]^{\frac{1}{n}}}			
					\right \},
\]
with a dimensional constant $C_n$.
\end{thm}

We use it to  to generalize the stability estimate  \cite[Proposition~2.4]{KN4} and \cite[Corollary~5.10]{KN1}. Let $\hbar(s) $ be the inverse function of $ \ka(s)$ and 
\[\label{eq:Ga}	
\Ga(s) \text{  the inverse function of } s^{n(n+2)+1} \hbar (s^{n+2}).
\]
Notice that
$	\Ga(s) \rightarrow 0 \quad \text{as } s \rightarrow 0^+.
$

\begin{prop}\label{prop:l1-stability} Let $\psi \in PSH(\omega) \cap C^0(X)$ and $\psi \leq 0$. Let $\mu \in \cF(X,h)$. Assume that $\vphi \in PSH(\om) \cap C^0(X)$ satisfies
$
	(\om +dd^c \vphi)^n = \mu .
$
Then, there exists a constant $C>0$  depending only on $\tau, \om$ and $\|\psi\|_\infty$ such that
\[\notag
	\sup_X(\psi - \vphi) \leq C\; \Ga\left(\left\|(\psi - \vphi)_+\right\|_{L^1(d\mu)} \right).
\]
\end{prop}

\begin{proof} Without loss of generality we may assume that $-1\leq \psi \leq 0$. Put \[ \notag U(\veps, s) = \{\vphi<(1-\veps) \psi + \inf_X [\vphi -(1-\veps) \psi] +s \},\]
where $0<\veps<1$ and $s>0.$

\begin{lem} 
\label{lem:cap-level-set}
For  
$0 <s \leq \frac{1}{3}\min\{\veps^n, \frac{\veps^3}{16 B} \}$, 
$0< t \leq \frac{4}{3} (1-\veps) \min\{\veps^n, \frac{\veps^3}{16 B} \}$ we have 
\[\notag
	t^n \, cap_{\om} (U(\veps, s))
	\leq 	 C F_h\left(cap_\om (U(\veps, s+t))\right),
\]
where $C$ is a dimensional constant.
\end{lem}

\begin{proof}[Proof of Lemma~\ref{lem:cap-level-set}]
By \cite[Lemma~5.4]{KN1} 
\[\label{eq:cap-growth}
	t^n \, cap_{\om} (U(\veps, s))
	\leq 	 C \,  \int_{ U (\veps, s + t) }\om_\varphi^n,
\]
The lemma  follows from the assumption on the measure $\om_\vphi^n = \mu$.
\end{proof}

To finish the proof of the proposition  we proceed as in \cite[Proposition~2.4]{KN4} or \cite[Theorem~3.11]{KN3}, though under a weaker assumption.  One needs  to estimate
$$-S:= \sup_X (\psi - \vphi) > 0$$ in terms of $\|(\psi - \vphi)_+\|_{L^1(d\mu)}$ as in the K\"ahler case \cite{K03}. Suppose that 
\[ \label{slr-eq2} \|(\psi - \vphi)_+\|_{L^1(d\mu)} \leq \de, 
\]
where $\de:= \veps^{n(n+2)+1} \hbar (\veps^{n+2})$. 
 Consider sublevel sets $U(\veps, t) = \{\vphi< (1-\veps) \psi + S_\veps +t \}$, where $S_\veps = \inf_X [\vphi -(1-\veps)\psi]$. 
It is clear that \[\notag S - \veps \leq S_\veps \leq S.\] 
Therefore, $U(\veps,2t) \subset \{\vphi < \psi + S+ \veps +2t\}$. Then, $(\psi - \vphi)_+ \geq |S| - \veps -2t>0$ for $0< t < \veps_B$ and $0< \veps < |S|/2$ 
on the latter set (if $|S| \leq 2 \veps$ then we are done).
By \eqref{eq:cap-growth} we have
\begin{align*}
	cap_{\omega}(U(\veps,t)) 
	\leq \frac{C}{t^n} \int_{U(\veps,2t)} d\mu
&	\leq \frac{C}{t^n} \int_X \frac{(\psi -\vphi)_+}{(|S| - \veps -2t)}
		d\mu\\
&	\leq \frac{C \|(\psi - \vphi)_+\|_{L^1(d\mu)}}{t^n (|S| - \veps -2t)} .
\end{align*}
Moreover, by the inequality \eqref{eq:kappa} it follows that 
$ \hbar(t)  \leq cap_\omega(U(\veps,t)).$
Combining these inequalities, we obtain
\[\notag
	(|S| - \veps - 2t) 
	\leq \frac{C \|(\psi - \vphi)_+\|_{L^1(d\mu)}}{t^n \hbar(t)} .
\]
Therefore, using \eqref{slr-eq2},
\begin{align*}
|S| 
&	\leq \veps + 2t 
	+ \frac{C \|(\psi - \vphi)_+\|_{L^1(d\mu)}}{t^n \hbar(t)}  \\
&	\leq 3 \veps + \frac{C \de}{t^n \hbar(t)}.
\end{align*}
Recall that $\veps_B = \frac{1}{3} \min\{\veps^n, \frac{\veps^3}{16B}\}$. So, taking 
$
	t = \veps_B/2 \geq \veps^{n+2} $
we have
\[\notag \frac{\de}{\veps^{n(n+2)} \hbar(\veps^{n+2})} =\veps.\] 
Notice that we used the fact that $\hbar (s)$ is also increasing.
Hence  $|S| \leq C \veps$ with $C = C(\omega)$. Thus,
\[\notag
	\sup_X(\psi -\vphi) \leq C \; \Ga \left(\|(\psi - \vphi)_+\|_{L^1(d\mu)}\right).
\]
This is the desired stability estimate.
\end{proof}

There is always a uniform lower bound for the volume of Monge-Amp\`ere measures dominated by capacity. This is  essentially \cite[Proposition~2.4]{KN2}.

\begin{prop}
\label{barrier-funct} Consider 
 $\mu\in \cF(X,h)$ such that $ \mu (X) >0$.
Suppose $w\in PSH(\omega) \cap C(X)$ and $c>0$  solve 
\[
	(\omega + dd^c w)^n = c\; \mu, \quad \sup_X w =0,
\]
Then there exists a constant $V_{min}>0$ depending only on $X, \omega, h$ such that whenever
\begin{equation}
\label{mass-c1}
		\int_X d\mu \leq  2V_{min},
\end{equation}
we have $c\geq 2^n$. 
\end{prop}

\begin{proof}
Suppose  $c \leq 2^n$. We shall see that this  leads to a contradiction for some positive $V_{min}$. Firstly, we  have
$
	\omega_w^n \leq 2^n  \mu.
$
Therefore,  the Monge-Amp\`ere measure $\omega_w^n$ satisfies the inequality  \eqref{eq:dominate} for the admissible  function 
$
	 h(x)/2^n.
$
The inequality \eqref{eq:cap-growth} for $0< t\leq  \frac{1}{3} \min\{\frac{1}{2^n}, \frac{1}{2^7 B}\}$ then gives:
\[\notag
	t^n cap_\omega(\{w < S + t\}) \leq C \int_{\{ w < S + 2t\}} \omega_w^n
	\leq C \int_X  2^n d\mu ,
\]
where $S= \inf_X w$ and $C>0$ depends only on $n,B$. It implies that
\begin{equation}
\label{mc-eq2}
	\frac{t^n}{2^n C} \, cap_\omega(\{w < S + t\}) \leq \int_X   d\mu.
\end{equation}
The formula \eqref{eq:kappa2} for the function $\ka_0(x)$  corresponding to $\omega_w^n$ is 
\[\notag
	\ka_0(s^{-n}) = 8 C_n \left\{\frac{1}{[h(s)]^\frac{1}{n}} 
			+ \int_s^\infty \frac{dx}{x [h(x)]^\frac{1}{n}} \right\}.
\]
It is defined on $(0, cap_\omega(X))$.
Since $\ka_0 (x)$ is an increasing function it has the inverse $\hbar_0(x)$. It follows from \eqref{eq:kappa} that for $0< t\leq  \frac{1}{3} \min\{\frac{1}{2^n}, \frac{1}{2^7 B}\}$ we have
\[\notag
	\hbar_0 (t) \leq cap_\omega (\{w < S +t\}). 
\]
Coupling this  with \eqref{mc-eq2} we obtain
\begin{equation}
\label{l1-apriori-bound}
	\int_X  d\mu \geq \frac{t^n \hbar_0(t)}{2^n C}.
\end{equation}
Define 
\begin{equation}
\label{v-min}
	V_{min}:= \frac{t_0^n}{2^{n+2} C\hbar_0(t_0)}>0, \quad 
	t_0 = \frac{1}{6} \min\{\frac{1}{2^n}, \frac{1}{2^7 B}\}.	
\end{equation}
Then,  \eqref{l1-apriori-bound} and the above choices lead to a contradiction
\[\notag
	2 V_{min} \geq \int_X  d\mu \geq 4 V_{min}>0.
\]
Thus the proposition is proven.
\end{proof}

\section{Existence of continuous solutions}

In this section we generalize  the results of \cite{KN1, KN4} on the existence of continuous solutions of the Monge-Amp\`ere equation.  
This is also the extension of \cite{K98,K05} from K\"ahler to Hermitian setting.  We prove 
the first theorem in the introduction.

\begin{thm}\label{thm:main} Let $\mu \in \cF(X,h)$ be such that  $\mu(X)>0$. Then, there exists a continuous $\om$-psh  function $u$ and a constant $c>0$ solving the equation
\[\notag
	(\om + dd^c u)^n = c\; \mu.
\]
\end{thm}


\begin{proof} The proof follows the scheme of the one in \cite[Theorem~1.3]{KN4}. We only clarify the differences. Let $\mu_\veps$ be the approximating sequence from Lemma~\ref{lem:app-seq}.  By \cite[Theorem~0.1]{KN1} there exist $u_\veps \in PSH(\om) \cap C^0(X)$ and a constant $c_\veps >0$ solving
$$	(\om + dd^c u_\veps)^n = c_\veps \mu_\veps, \quad \sup_X u_\veps =0. 
$$	
The main difficulty lies in proving the uniform upper bound for constants $\{c_\eps\}$ which requires a bit different approach compared to \cite{KN1, KN4}. 

Since  $\lim_{\veps \to 0} \mu_\veps(X) = \mu(X)$, there exist a ball $U= B(a,s) \subset U' = B(a, 2s)$ in the finite open cover \eqref{eq:cover} and a positive constant $C_1>0$ such that
\[\label{eq:uni-low}	\mu_{U} * \vrh_
\veps (X) =  \mu_{U} * \vrh_
\veps (U') \geq  C_1 \quad 
\]
for every small $\veps>0$, where we recall  $\mu_U$ is the restriction of $\chi_a \mu$ to $U$, and $\chi_a$ is the smooth function in the partition of unity subordinate to $\{B(x_j, s)\}$. 
Let us denote $\Om= B(a,3s)$. Thanks to \cite{CKNS85} there is  $v_\veps \in PSH(\Om) \cap C^\infty(\bar \Om)$ such that  $(dd^c v_\veps)^n =  \mu_U * \vrh_\veps + \veps \om^n$ and $v_\veps = 0$ on $\d\Om$. 
By Lemma~\ref{lem:app-seq} and \cite{Ko96} it follows that 
\[\label{eq:uniform-v}	\| v_\veps \|_{L^\infty(\Om)} \leq C_2 = C(\Om, h_0).
\]
It is clear that $\mu_\veps \geq \mu_U * \vrh_\veps$ on $\Om$. Let us  write 
$\mu_U * \vrh_\veps  = R_\veps \om^n$ for a smooth function $R_\veps$ in $\Om$. Using the mixed forms type inequality \cite[Lemma~2.2]{KN2} we have
$$\begin{aligned}
	\om_{u_\veps} \wed (dd^c v_\veps)^{n-1} 
&\geq 	\left(\frac{\om_{u_\veps}^n}{(dd^cv_\veps)^n}\right)^\frac{1}{n} (dd^c v_\veps)^n  \\
&\geq		\left(\frac{c_\veps R_\veps}{R_\veps + \veps}\right)^\frac{1}{n}	(R_\veps + \veps) \om^n \\
&\geq	c_\veps^\frac{1}{n} R_\veps \om^n.
\end{aligned}$$
Therefore,
\[\label{eq:low-inq}
	\int_{\Om'} \om_{u_\veps} \wed (dd^c v_\veps)^{n-1} \geq c_\veps^\frac{1}{n} C_1.
\]
Fix a strictly plurisubharmonic function $\rho_\Om$ in $\Om$ such that $\om \leq dd^c \rho_\Om$. Then  the Demailly's version of the Chern-Levine-Nirenberg inequality \cite{De85} gives
$$\begin{aligned}
	\int_{\Om'} \om_{u_\veps} \wed (dd^cv_\veps)^{n-1} 
&\leq 	\int_{\Om'} dd^c (u+\rho_\Om) \wed  	 (dd^cv_\veps)^{n-1} \\
&\leq 	C(\Om',\Om) \|v_\veps\|_{L^\infty(\Om)}^{n-1} \left( \|\rho_\Om\|_{L^1(\Om)} + \|u_\veps\|_{L^1(\Om)} \right) \\
\end{aligned}$$
Notice that $\int_X |u_\veps| \om^n$ is uniformly bounded (see e.g. \cite[Proposition~2.5]{DK12}). These combined with \eqref{eq:uniform-v} and \eqref{eq:low-inq}  give  the uniform upper bound for $\{c_\veps\}_{\veps>0}$. The uniform lower bound of this sequence follows from \cite[Lemma~5.9]{KN1} as $
\mu_\veps(X)$ is uniformly bounded. By the proof of \cite[Corollary~5.6]{KN1} it follows
\[\notag
	\|u_\veps\|_{L^\infty(X)} < C.
\]

Now we continue as in the proof of \cite[Theorem~1.3]{KN4}. 
Since the sequence $\{u_\veps\}_{\veps>0}$ normalized by $\sup_X u_\veps =0$ is a compact subset of $L^1(X)$, passing to a subsequence, we may assume  that
\[\label{eq:l1-cov}
	 u_\veps \longrightarrow u \mbox{ in }L^1(X);
\]
moreover, $u\in PSH(\om) \cap L^\infty(X)$ and also $\lim_{\veps \to 0} c_\veps = c>0.$

We wish to apply Proposition~\ref{prop:l1-stability} to conclude that  the convergence \eqref{eq:l1-cov} is in $C^0(X).$ This amounts  to showing that 
\[\label{eq:converge-e}
	\lim_{ \veps \to 0} \int_X |u_\veps -u| d \mu_\veps =0.
\]
By \eqref{eq:dominate} the measure $\mu$ satisfies  
\[\notag
	\mu(K) \leq A cap_\om(K) 
\]
for all Borel sets  $K\subset X$, where $A$ is a uniform constant. Furthermore, the potentials are uniformly bounded, so we can repeat the arguments of \cite[Lemma~4.4]{GZ07}  (see also \cite[Lemma~5.2 and Proof of Theorem~5.1]{Ce98}) to finish the proof of \eqref{eq:converge-e}. 
Finally, we get that $u_\veps$ converges to $u$ in $C^0(X)$, which is a solution to  $\om_u^n = c \mu$. 
\end{proof}

\begin{cor} \label{cor:good-stability} Suppose that $\mu_j \in \cF(X, h)$ and it is smooth  for every $j\geq 1$. Let $\mu_j$ converge weakly to $\mu \in \cF(X, h)$ as 
$j\to +\infty$. For  $j\geq 1$ let us solve
$$	(\om+ dd^c v_j)^n = e^{v_j} \mu_j.
$$ Then  $v_j$ converges uniformly to a continuous $\om$-psh function $v$ as $j\to +\infty$. Consequently, $v$  is  the unique continuous $\om$-psh solution to $\om_v^n = e^v\mu$. 
\end{cor}

\begin{proof} With the estimates \eqref{eq:uni-low}, \eqref{eq:uniform-v} and \eqref{eq:low-inq} at hand the proof of \cite[Theorem~2.1]{N16} is readily adaptable to  this setting which gives the existence of a continuous solution. The uniqueness follows from \cite[Lemma~2.3]{N16} with the same proof.
\end{proof}

\begin{cor} \label{cor:max}Let $\mu \in \cF(X, h)$ and $\la>0$. Then, there exists a unique continuous $\om$-psh solution $v$ to 
$$	(\om+ dd^c v)^n = e^{\la v} \mu.
$$
\end{cor}

\begin{proof} It is a simple application of Corollary~\ref{cor:good-stability} for the approximating sequence $\mu_\veps$   from Lemma~\ref{lem:app-seq}.
\end{proof}

\section{stability of solutions}
We prove a stability estimate for measures belonging  to $\cF(X,h)$  which are strictly positive, absolutely continuous with respect to the Lebesgue measure. 
We use the following notation: the $L^p$-norms for $0<p<\infty$  are 
$$
\|\cdot\|_p := \left(\int_X |\cdot|^p \om^n\right)^\frac{1}{p} \quad\text{and}\quad \|\cdot\|_\infty := \sup_X |\cdot|.
$$

\begin{thm} \label{thm:x-stability} Assume $f,g \in L^1(X)$ and $f\om^n, g\om^n\in \cF(X,h)$. Consider two bounded $\om$-psh solutions $u,v$ of 
$$\begin{aligned}
	\om_u^n = f\om^n, \quad 
	\om_v^n = g\om^n
\end{aligned}$$
with $\sup_X u =  \sup_X u =0.$
Suppose that $f\geq c_0>0$. Fix $\ga> 2+ n(n+1)$. Then, 
$$
	\|u-v\|_{\infty} \leq C \veps
$$
provided that
$$ \|f-g\|_{1} \leq \hbar(\veps^{n+1}) \veps^\ga.$$
\end{thm}

We will adapt the proof of \cite[Theorem~3.1]{KN2} with necessary changes. First, it is enough to assume that  $f,g$ are smooth. 

\begin{lem} Let $f_j, g_j \in \cF(X,h)$ be smooth sequences of functions converging in $L^1(X)$ to $f, g$ respectively. Let $u_j, v_j \in PSH(\om) \cap C^\infty(X)$ be such that $u_j \searrow u$ and $v_j \searrow u$ as $j\to +\infty$. Assume $\vphi_j,\psi_j$ solve 
$$	(\om+ dd^c \vphi_j)^n= e^{\vphi_j - u_j} f_j \om^n, \quad (\om+dd^c \psi_j)^n= e^{\psi_j-v_j} g_j \om^n.
$$
Then, 
$$ \|u-v\|_{\infty} = \lim_{j\to +\infty} \|\vphi_j - \psi_j\|_{\infty}.
$$
\end{lem}

\begin{proof} We use the argument of \cite[Remark~3.11]{KN2} pointed out by a referee of that paper. By Corollary~\ref{cor:good-stability} the sequence $\{\vphi_j\}$ converges uniformly to the solution $u_0$ of $$(\om+ dd^c u_0)^n = e^{u_0} \left(e^{-u} f \right) \om^n.$$ It follows from the uniqueness of $u$ that $u_0= u$. Similarly, $\{\psi_j\}$ converges uniformly to $v$. The conclusion follows.
\end{proof}

\begin{proof}[Proof of Theorem~\ref{thm:x-stability}]
We fix the notation as in the proof of \cite[Theorem~3.1]{KN2}. For $t\in \bR$ define
\[\notag
	\varphi = u-v, \quad \Omega(t) = \{\varphi<t\}, \quad t_0 = \inf_X \varphi.
\]
We need to replace \cite[Lemma~3.4]{KN2} by the following statement.  The proof is similar up to some technicalities. For the reader's convenience we give  all details here.

\begin{lem}
\label{case1-l1}
Let $V_{min}>0$ be the constant in Proposition~\ref{barrier-funct}. Fix $t_1 > t_0$.
Assume that for $0<\varepsilon <<1$,
\[\notag
	\|f-g\|_1 \leq \ell(\veps) \veps,
\]
where $\ell(\veps)=\hbar(\veps^{(n+1)\al})$. If $\int_{\Omega(t_1)} f \omega^n \leq V_{min}$, then 
$$
	t_1 - t_0 \leq C \varepsilon^\alpha 
$$
where $0< \alpha <\frac{1}{2+ n(n+1)}$ is fixed.
\end{lem}

\begin{proof} Define the sets:
\[\notag
\Omega_1 := \{z\in \Omega(t_1): f (z) \leq (1+ \varepsilon^\alpha) g (z)\} \quad
\mbox{and} \quad \Omega_2:= \Omega(t_1)\setminus \Omega_1.
\]
Since $g < \varepsilon^{-\alpha} (f-g)$ on $\Omega_2$, we have
\begin{equation}
\label{c1-l1-eq4}
\begin{aligned}
	\int_{\Omega_2} f \omega^n
&	\leq 		\int_{\Omega_2} |f-g| \omega^n + \int_{\Omega_2} g \omega^n \\
&	\leq		\ell(\veps)\veps + \ell(\veps)\veps^{1-\al} \\
&	\leq 2 \ell (\veps) \veps^{1-\alpha}.
\end{aligned}
\end{equation}
It follows that 
\begin{align*}
	\int_{\Omega(t_1)} f \omega^n
	=	\int_{\Omega_1} f \omega^n + \int_{\Omega_2} f \omega^n
	\leq	\int_{\Omega_1} f \omega^n+ 2 \ell (\veps) \veps^{1-\alpha}
	\leq V_{min} + 2 \ell (\veps) \veps^{1-\alpha}.
\end{align*}
Next, we  construct a barrier function by putting
\begin{equation} \label{c1-l1-eq6}
\hat f (z)  =	\begin{cases}
		f (z)\quad &\mbox{for } z\in \Omega(t_1), \\
		\frac{1}{A}f(z)\quad &\mbox{for }  z\in X \setminus \Omega(t_1).
	\end{cases}
\end{equation}
As $\int_{\Omega(t_1)} f\omega^n \leq V_{min}$ we can choose $A>1$ large enough so that 
\[\notag
	\int_X \hat f \omega^n\leq \frac{3}{2} V_{min} .
\]
Notice that $f/A \leq \hat f \leq f$. By Theorem~\ref{thm:main} we find $w \in PSH(\omega) \cap C(X)$ and $\hat c>0$ satisfying
\[\notag
	(\omega + dd^c w)^n = \hat c \hat f \omega^n, \quad \sup_X w =0 .
\]
By  Proposition~\ref{barrier-funct} applied for $f,h$ we have 
\begin{equation}\label{c1-l1-eq9}
	2^n \leq \hat c \leq A ,
\end{equation}
where the last inequality follows from \eqref{c1-l1-eq6} and \cite[Lemma~2.1]{KN2}.
Hence, 
\begin{equation}\label{c1-l1-eq10}
	\hat c\hat f \geq 2^n f \quad \mbox{on } \Omega(t_1).
\end{equation}
Define for $0< s <1$, 
$
	\psi_s = (1-s) v + s w.
$
It follows from the mixed forms type inequality (\cite[Lemma~2.2]{KN2}) that 
\begin{equation*}
\begin{aligned}
	(\omega + dd^c \psi_s)^n 
&	\geq 		\left[ (1-s) g^\frac{1}{n} + s (\hat c \hat f/f)^\frac{1}{n}\right]^n \omega^n \\
&	=[(1-s) (g/f)^\frac{1}{n} + s(\hat c \hat f/f)^\frac{1}{n} ]^n f\omega^n \\
&	=: [b(s)]^n f\omega^n. 
\end{aligned}
\end{equation*}
Therefore  on $\Omega_1$  we have
\[\notag
	b(s) \geq \frac{(1-s)}{(1+\varepsilon^\alpha)^\frac{1}{n}} + 2 s 
	\geq \frac{1-s}{1+\varepsilon^\alpha} + 2s.
\]
If $2 \varepsilon^\alpha \leq  s \leq 1$, then 
\begin{equation}\label{c1-l1-eq14}
	b(s) \geq 1 + \varepsilon^\alpha \quad \mbox{on } \Omega_1. 
\end{equation}
Let us use the notation
$
	m_s:= \inf_X (u - \psi_s) = \inf_X \{u-v + s(v - w)\} .
$
Then,
\[\notag
	m_s \leq t_0 + s \|w\|_\infty .
\]
Set for $0<\tau<1$, 
$
	m_s(\tau) := \inf_X [u - (1-\tau)\psi_s] .
$
Then
$
	m_s(\tau) \leq m_s.
$
By the above definitions we have
\begin{equation}\label{c1-l1-eq19}
\begin{aligned}
U(\tau, t)
& :=	\{u< (1-\tau)\psi_s + m_s(\tau) + t \}  \\
&	\subset 	\{u < \psi_s + m_s + \tau \|\psi_s\|_\infty + t\} \\
&	\subset	\{u<v + t_0 + s(\|v\|_\infty + \|w\|_\infty) + \tau \|\psi_s\|_\infty+ t \} .
\end{aligned}
\end{equation}
We are going to show that
\begin{equation}\label{c1-l1-eq20}
	t_1 - t_0 \leq 2s(\|v\|_\infty + \|w\|_\infty) + \tau \|\psi_s\|_\infty,
\end{equation}
for $s= 2\varepsilon^\alpha$ and $\tau = \varepsilon^\alpha/2$.  Suppose  it was false. 
By \eqref{c1-l1-eq19} we have
\begin{align*}
	U(\tau, t)  \subset \subset	\{u< v + t_0 + (t_1-t_0)\} = \Omega(t_1), 
\end{align*}
for $0<t < \frac{t_1 - t_0}{2}$. 
To go further we need to estimate the integrals:
\[\notag
	\int_{U(\tau,t)} f\omega^n
\]
for $0< t << s, \tau$.
By the modified comparison principle \cite[Theorem 0.2]{KN1} 
\[\notag
	\int_{U(\tau, t)} \omega_ {(1-\tau) \psi_s}^n 
\leq \left(1 + \frac{C t}{\tau^n}\right) \int_{U(\tau, t)} \omega_u^n,
\]
for every $0< t < \min\{\frac{\tau^3}{16B}, \frac{t_1-t_0}{2}\}$. Hence, a simple estimate from below gives 
\[\notag
	(1-\tau)^n \int_{U(\tau,t)} \omega_{\psi_s}^n \leq 
\left(1 + \frac{C t}{\tau^n}\right) \int_{U(\tau, t)} \omega_u^n.
\]
Using \eqref{c1-l1-eq14} for $s=2\varepsilon^\alpha$ we get
\begin{equation}\label{c1-l1-eq25}
	(1-\tau)^n (1+\varepsilon^\alpha)^n \int_{U(\tau, t)\cap \Omega_1} f \omega^n
\leq 	\left(1 + \frac{C t}{\tau^n}\right) \int_{U(\tau, t)} f \omega^n.
\end{equation}
If we write $a(\varepsilon, \tau) = (1-\tau)^n (1+\varepsilon^\alpha)^n$, then
\[\notag
	a(\varepsilon, \tau) = (1+\varepsilon^\alpha/2 - \varepsilon^{2\alpha}/2)^n
	>1 + \varepsilon^\alpha/4
\]
as we have $\tau = \varepsilon^\alpha/2$ and $0<\varepsilon^\alpha<1/4$. Therefore  \eqref{c1-l1-eq25} implies that
\[\notag
	\left[ a (\varepsilon, \tau) - \left(1 + \frac{2^nC t}{\varepsilon^{n\alpha}}\right) 
	\right] \int_{U(\tau,t)\cap \Omega_1} f \omega^n 
\leq  		\left(1 + \frac{2^nC t}{\varepsilon^{n\alpha}}\right)  
		\int_{\Omega_2} f \omega^n .
\]
Thus  for $0< t \leq \varepsilon^{(n+1)\alpha}/2^{n+3}C$,
\[\notag
	\frac{\varepsilon^\alpha}{8} \int_{U(\tau,t)\cap \Omega_1} f \omega^n 
\leq		2 \int_{\Omega_2} f \omega^n 
\leq 		4 \ell (\veps) \veps^{1-\alpha},
\]
where the last inequality used \eqref{c1-l1-eq4}. Hence,
\begin{equation*}
	 \int_{U(\tau,t)\cap \Omega_1} f  \omega^n \leq 32 \, \ell (\veps) \veps^{1-2\alpha}.
\end{equation*}
Altogether we get that for $0< t \leq \varepsilon^{(n+1)\alpha}/C$, 
\begin{equation}\label{c1-l1-eq30}
\int_{U(\tau, t)} f \omega^n 	
\leq	\int_{U(\tau,t)\cap \Omega_1} f  \omega^n +
	\int_{\Omega_2} f \omega^n
\leq 	C \ell (\veps) \veps^{1-2\alpha} .
\end{equation}
This is the estimate  we need. 

Now we are able  make use of the results from \cite{KN1} recalled above. 
First, it follows from \eqref{eq:kappa} and \eqref{eq:cap-growth} that for $0< t \leq \varepsilon^{(n+1)\alpha}/C$, 
\begin{equation}\label{c1-l1-eq31}
	\hbar (t/2) \leq cap_\omega(U(\tau,t/2)) \leq \frac{2^nC}{t^n} \int_{U(\tau, t)} f\omega^n,
\end{equation}
where $\hbar(t)$ is the inverse of $\kappa(t)$. 
It follows from \eqref{c1-l1-eq30} and \eqref{c1-l1-eq31} that 
\[\notag
	\hbar(t) \leq \frac{C\ell (\veps) \veps^{1-2\alpha}}{t^n} .
\]
Then, taking $t =\varepsilon^{(n+1)\alpha}$ we obtain that
\[\notag
	\hbar (\varepsilon^{(n+1)\alpha}) \leq C \hbar(\veps^{(n+1)\al})\varepsilon^{1-2\alpha - n(n+1)\alpha}. 
\]
Equivalently, $1 \leq C \veps^{1-2\alpha - n(n+1)\alpha}$. 
However,  we have that
$
	 1 - [n(n+1)+2] \alpha>0,
$
which leads to a contradiction 
for $\varepsilon>0$ small enough. 

Thus we have proved that 
\[\notag
	t_1 - t_0 \leq 4 \varepsilon^{\alpha} (\|v\|_\infty + \|w\|_\infty + \|\psi_s\|_\infty),
\]
for a fixed $0< \alpha < \frac{1}{2 + n(n+1)}$. The  norms on the right hand side are  controlled by $\|f\|_1, \|g\|_1, h, V_{min}$. So the  lemma follows by rewriting $\ga= 1/\al$ and $\veps:= \veps^{1/\al}$.
\end{proof}

Thanks to the above lemma, the remaining part of the proof of  \cite[Theorem~3.1]{KN2} is used to conclude that of Theorem~\ref{thm:x-stability}. 
\end{proof}

\section{The Dinh-Nguyen theorem on Hermitian manifolds}

In this section we give a characterization of measures leading to H\"older continuous solutions of the  Monge-Amp\`ere equation on compact Hermitian manifolds, which
is an analogue  of the Dinh-Nguyen theorem \cite{DN14}. If $\om$ is K\"ahler, \cite{DN14}  says that a positive Radon measure admits a H\"older continuous $\om$-psh potential if and only if the associated functional is H\"older continuous on $\{w \in PSH(\om): \sup_X v=0\}$ with respect to the $L^1$-distance.  
Let us denote
\[\notag
	\cS= \cS(\om) := \left\{ u\in PSH(\om): -1\leq u\leq 0, \;\sup_X u=0\right\}.
\]
The  $L^1$-distance, with respect to the Lebesgue measure,   between  $u,v \in PSH(\om)$  is given by
\[
	\| u-v \|_{L^1} := \int_X |u-v | \om^n.
\]
A measure $\mu$ gives the natural functional $\hat\mu: PSH(\om) \to \bR$ defined by
$$	\hat{\mu} (v)= \int_X v d\mu.
$$
Following Dinh-Nguyen \cite{DN14} 
we say that 
\begin{defn}
 $\hat{\mu }$ is H\"older continuous on $\cS$ if it is  H\"older continuous  with respect to the $L^1$ distance.
\end{defn}
In other words there exist a uniform exponent $\al>0$ and a uniform constant $C>0$ such that  for every $u,v \in \cS$,
\[
	|\hat{\mu }(u -v)| = \left| \int_X (u-v) d\mu \right| \leq C \|u-v\|_{L^1}^\al.
\] 
Since $\max\{u,v\}\in \cS$ for every $u, v\in \cS$, this inequality is equivalent to
\[ \label{eq:global-h}
	\int_X |u-v| d\mu \leq C \|u-v\|_{L^1}^\al \quad \forall u,v \in \cS.
\]

We are going to show that the H\"older continuity property on $\cS$ is  local. Let $\Om$ be  a strictly pseudoconvex domain  in $\bC^n$ and define
\[
	 \cS_0(\Om) := \{v\in PSH(\Om): -1\leq v\leq 0\}.
\]
The $L^1$ distance (with respect to the Lebesgue measure) between $\vphi, \psi \in \cS_0$ is defined similarly:
$$
	 \|\vphi-\psi\|_{L^1(\Om)} = \int_{\Om} |\vphi-\psi| dV_{2n}. 
$$
Let $\nu$ be a positive Borel measure on $\Om$. It also gives a natural functional $\hat{\nu}$ on $ PSH(\Om)$ defined by
$$
	\hat{\nu }(\vphi)= \int_\Om \vphi d\nu. 
$$ 
\begin{defn}
	$\hat{\nu }$ is locally H\"older continuous on $\cS_0(\Om)$   if for a fixed $\Om' \subset \subset \Om$, there exists a constant $C = C(\Om',\Om)>0$ and an exponent $\al>0$ such that for every $\vphi, \psi \in \cS_0(\Om)$
	\[\label{eq:local-h}
	\int_{\Om'} |\vphi- \psi| d\nu  \leq C \|\vphi-\psi\|_{L^1(\Om)}^\al.
	\]
\end{defn}
 
\begin{lem}\label{lem:equiv-loc-glo} Let $\mu$ be a positive Borel measure on $X$. Then,  $\hat{\mu }$ is H\"older continuous on $\cS$ if and only if it is locally H\"older continuous on every  local coordinate chart. 
\end{lem}

\begin{proof} Suppose that $\hat{\mu}$ is locally H\"older continuous on each local coordinate chart. Let $u,v \in \cS$. We wish to show that there exist $C,\al>0$ such that
\[\notag
	\int_X |u-v| d\mu \leq C \|u-v\|_{L^1}^\al.
\] 
Let $B (a,r)$  be a local coordinate ball in the finite covering \eqref{eq:cover}. Let $\rho$ be a strictly plurisubharmonic function on $U:=B(a,2r)$ such that $dd^c \rho \geq \om$. Define  $\vphi:= u+\rho, \psi:= v+\rho.$ By local H\"older continuity of 
$\hat{\mu}$ we have
\[\notag
	\int_{B(a,r)} |u-v| d\mu = \int_{B(a,r)} |\vphi-\psi| d\mu \leq C \|\vphi - \psi \|^\al_{L^1(U)}  \leq C \left(\int_X |u-v| \om^n \right)^\al.
\]
Summing up over all  $j\in J$ of the cover, we get that $\hat{\mu}$ is H\"older continuous on $\cS$. 

For the reverse direction, assume now  that $\hat{\mu}$ is  H\"older continuous $\cS$. Let $B(a,r), U$ be the coordinate balls  above. Take $\vphi, \psi \in \cS_0(U)$. Let $\chi$ be a $\om$-psh function on $X$ such that $\chi =0$ outside $U$ and $\chi\leq -3\de$ on $B(a,r)$ for some $0<\de <1/2$. Define
\[\notag
	\tilde \vphi = \begin{cases} 
		\max\{\de \vphi - \de, \chi\} \quad&\mbox{on } U, \\
		\chi	\quad &\mbox{on } X\setminus U,
	\end{cases}
\]
and $\tilde \psi$ analogously. Then, using the assumption
\[\notag	
	\de \int_{B(a,r)} |\vphi - \psi| d\mu \leq \int_X |\tilde \vphi - \tilde \psi| d\mu \leq C \left( \int_X |\tilde \vphi - \tilde \psi| \om^n\right)^\al  = C \left( \int_U |\tilde \vphi - \tilde \psi| \om^n\right)^\al. 
\]
Note that on $U$ we have $|\tilde \vphi - \tilde \psi| \leq \de |\vphi - \psi|$. It follows that
$$
	\int_{B(a,r)} |\vphi -\psi| d\mu \leq \frac{C}{\de^{1-\al}}  \left( \int_U  |\vphi - \psi| \om^n\right)^\al .
$$
This is the local H\"older continuous property  of $\hat{\mu}$ on $U$.
\end{proof}

There are plenty of examples of measures  which are locally H\"older continuous on $\cS_0(\Om)$ (see \cite{N17,N18}).  We give below a sufficient condition. Let us consider the class
$$ \cE_0'(\Om) = \left\{ v\in PSH \cap L^\infty (\Om): \lim_{z\to \d\Om} v(z) =0, \int_{\Om}(dd^cv)^n \leq 1\right\}.
$$
Then, the H\"older continuity of a functional on $\cE_0'(\Om)$ is considered with respect to $L^1$-distance \cite[Definition~2.3]{N17}.

\begin{lem} If $\hat{\nu}$ is H\"older continuous on $\cE_0'(\Om)$, then it is locally H\"older continuous on $\cS_0(\Om)$.
\end{lem} 

\begin{proof} Let $\Om' \subset \subset \Om$ and $u,v\in \cS_0(\Om)$. Let $\rho$ be the defining function of $\Om$. By the maximum construction we may assume that there are $\tilde u, \tilde v \in PSH(\Om)$ such that
$$	\tilde u = u, \tilde v =v \quad \text{in } \Om'
$$ 
and $\tilde u = \tilde v = \rho$ near $\d\Om$. The Chern-Levin-Nirenberg inequality implies that $\tilde u/c_0, \tilde v/c_0 \in \cE_0'(\Om)$ for a constant $c_0>0$ depending only on $\rho$ and $\Om',\Om$. Thus,
$$
	\int_{\Om'} |u-v| d\nu \leq \int_\Om |\tilde u- \tilde v| d\nu \leq  C c_0 \|\tilde u - \tilde v\|^\al_{L^1(\Om)} \leq  C c_0 \|u-v\|_{L^1(\Om)}^\al,
$$ 
where the last inequality used the fact that $|\tilde u -\tilde v| \leq |u-v|$ in $\Om$.
\end{proof}

Let us consider the following classes of measures:
$$
\cH(\tau) = \left\{ \mu \in \cF(X,h_1): h_1(x) = C_1 x^{n\tau} \text{ for some } C_1,\tau>0\right\},
$$
and the moderate measures, which by definition, are  in $\cF(X,h_2)$ with $h_2(x) = C_2 e^{\al x}$ for some $C_2,\al>0$.
For the latter the stability estimate of its potential has a nicer form, i.e., the function defined in \eqref{eq:Ga} is
$$	\Ga(s) = C s^\al \quad \text{with } \al>0.
$$
We observe that the proof of \cite[Proposition~4.4]{DN14} holds true for a general Hermitian metric $\om$. This gives a sufficient condition for moderate measures.

\begin{prop}\label{prop:dn-moderate} If $\hat{\mu}$ is H\"older continuous on $\{v\in PSH(\om): \sup_X v=0\}$, then it is moderate.
\end{prop}

Another sufficient condition for a measure to be  moderate, due to Dinh, Nguyen and Sibony \cite{DNS}, is as follows.  

\begin{lem}\label{lem:holder-sub} If there exists a H\"older continuous $\om$-psh function $\vphi$ and a constant $C>0$ such that 
$
	\mu \leq C \om_\vphi^n,
$
then $\mu$ is moderate and  $\hat{\mu}$ is H\"older continuous on $\cS$.
\end{lem}

\begin{proof}  These properties are local by  \cite[Lemma~1.2]{KN4} and Lemma~\ref{lem:equiv-loc-glo}. Therefore, we only prove them in a local coordinate chart. Let $U:= B(x,r) \subset \Om:= B(x,2r)$. Then, we can assume $\mu$ is compactly supported in $U$ and $\mu \leq (dd^c\vphi)^n$ for some H\"older continuous plurisubharmonic function $\vphi$ in $\Om$. By \cite[Corollary~1.2]{DNS} (see also \cite[Lemma~2.7, Proposition~2.9]{N17}) we get  that $\mu$ is moderate and $\hat{\mu}$ is  H\"older continuous on $\cE_0'(\Om)$. Thus, it is also H\"older continuous on $\cS$.
\end{proof}

\begin{remark} If $\om$ is K\"ahler, then  under the assumption of the lemma  $\mu$ is indeed H\"older continuous on  $\{v\in PSH(\om): \sup_X v=0\}$. However, due to the torsion terms $dd^c \om$ and $d\om \wed d^c \om$ in the general Hermitian case, it seems the H\"older continuity only holds on the smaller set $\cS$.
\end{remark}

We are ready to prove the Dinh-Nguyen type characterization on Hermitian manifolds.

\begin{thm} \label{thm:dn-chac} A positive Radon measure $\mu$ belongs to $\cH(\tau)$ and $\hat{\mu}$ is H\"older continuous on $\cS$ if and only if  there exists a H\"older continuous $\om$-psh function $u$ and a constant $c>0$ such that 
\[
	(\om + dd^c u)^n = c \;\mu.
\]
\end{thm}

\begin{proof}  The second condition implies the first by Lemma~\ref{lem:holder-sub}. It remains to show the reverse direction.
Theorem~\ref{thm:main} gives a continuous $\om$-psh function $u$ and a constant $c>0$  solving the equation. To show that the function $u$ is H\"older continuous we follow  the proof of \cite[Theorem~1.3]{KN4}. Note that we used the H\"older continuity of $\hat{\mu}$ on $\cS$  and \cite[Eq.(1.1)]{K08} to get the validity of \cite[Lemma~2.8]{KN4} in the present setting. 
\end{proof}

The last theorem allows to  extend   results of Pham \cite{hiep} and  Vu \cite{viet16} from the  K\"ahler to the Hermitian setting.

\begin{prop} Let $\mu$ be a positive Radon measure on $X$. Assume there exist constants $A,\al , t_0 >0$ such that for every  ball $B(x,t) \subset X$, 
$$
	\mu(B(x,t)) \leq A t^{2n-2+\al} \quad\text{for every } 0<t \leq t_0. 
$$ Suppose $0\leq f \in L^p(X,d\mu)$ with $p>1$. Assume that  $\int_X f d\mu>0$. Then, there exist a constant $c>0$ and a H\"older continuous $\om-$psh function solving
$$	(\om+ dd^c u)^n = c f d\mu.
$$
\end{prop}

\begin{proof} By Theorem~\ref{thm:dn-chac} it is sufficient to show that $fd\mu$ belongs to $\cH(\tau)$ for some $\tau>0$ and  that the corresponding functional is H\"older continuous on $\cS$. These properties are local. We may assume that $\supp \mu \subset U:= B(a,r) \subset \Om := B(a,2r)$ in $\bC^n$. By \cite[Lemma~2.15, Corollary~2.14]{N17} it follows that $\hat{\mu}$ is H\"older continuous on $\cE_0'(\Om)$, then so is the functional of $fd\mu$. Finally, by \cite[Propositon~2.9]{N17} we have  that $fd\mu$ is moderate. 
\end{proof}

One example of measures satisfying the assumption of the proposition above is given by the smooth volume form of a smooth hypersurface as in Pham \cite{hiep}. 

\begin{cor} Let $S$ be a compact smooth real hypersurface in $X$ and $dV_S$ is its smooth volume form. Then, for every $0\leq f \in L^p(S, dV_S)$ with $p>1$ and $\int_S f dV_S >0$, there exist a constant $c>0$ and a H\"older continuous $\om-$psh function $u$ solving 
$$	(\om + dd^c u)^n = c  f dV_S.
$$
\end{cor}

Later on, Vu \cite{viet16} proved the result for a  generic CR immersed $C^3-$submanifold of $X$. The K\"ahler assumption in his paper is needed only to use the characterization of \cite{DN14}. Given our results above we get immediately the statement of his result in the Hermitian setting. Actually, we can also simplify a bit his arguments by using the local H\"older continuity criterion (Lemma~\ref{lem:equiv-loc-glo}).

\bigskip
\bigskip

\end{document}